\documentclass[12pt, reqno]{amsart}
\usepackage{amsmath, amsthm, amsxtra, amscd, amsfonts, amssymb, mathrsfs, graphicx, color,ulem, pstricks}

\usepackage{pgfplots}
\pgfplotsset{compat=1.15}
\usepackage{mathrsfs}
\usetikzlibrary{arrows}

\usepackage[bookmarksnumbered, colorlinks, plainpages]{hyperref}
\hypersetup{colorlinks=true,linkcolor=red, anchorcolor=green, citecolor=cyan, urlcolor=red, pdftoolbar=true}
\usepackage{setspace}
\theoremstyle{plain}
\newtheorem{theorem}{Theorem}[section]
\newtheorem{lemma}[theorem]{Lemma}
\newtheorem{proposition}[theorem]{Proposition}

\theoremstyle{definition}

\newtheorem{example}[theorem]{Example}

\theoremstyle{remark}
\newtheorem{remark}[theorem]{Remark}
\numberwithin{equation}{section}

\definecolor{darkgreen}{rgb}{.1,.5,0}

\theoremstyle{plain}

\theoremstyle{definition}

\newtheorem*{Sketch of proof}{Sketch of proof}

\numberwithin{equation}{section}
\begin{document}
\setcounter{page}{1}
\title[ completion problem for extension of $m$- isometric weighted composition operators ]{completion problem for extension of $m$- isometric weighted composition operators on directed graphs}

\author[ V. Devadas , T.  Prasad,  E. Shine Lal ]{ V. Devadas , T.  Prasad and E. Shine Lal }

\address{V. Devadas \endgraf
		Department of Mathematics, Sree Narayana College , Alathur , 
		Kerala, 
		India -678682}

	\email{\textcolor[rgb]{0.00,0.00,0.84}{ v.devadas.v@gmail.com}}
\address{E. Shine Lal  \endgraf
		Department of Mathematics, University College, Thiruvananthapuram,
		Kerala, 
		India- 695034.}

	\email{\textcolor[rgb]{0.00,0.00,0.84}{ shinelal.e@gmail.com}}

\address{T. Prasad\endgraf
		Department of Mathematics, 
		University of Calicut,
		Kerala-673635, 
		India.}
	
	\email{\textcolor[rgb]{0.00,0.00,0.84}{ prasadvalapil@gmail.com}}


\begin{abstract}
In this paper, we discuss $k$-quasi-$m$-isometric completion problem of unilateral weighted shifts and  composition operators on directed graphs with one circuit and more than one branching vertex.
\end{abstract}
\maketitle
\section{Introduction and Preliminaries}

When a class of bounded linear operators $\mathcal{C}$ and a finite sequence of positive weights are given, an important problem in operator theory is to determine whether this sequence can be extended to the weight sequence of an injective, bounded unilateral weighted shift belonging to the class $\mathcal{C}$. Such an extended sequence is called a completion of the original sequence in the class $\mathcal{C}$. A considerable amount of research in the literature deals with the subnormal completion problem, as well as completion problems for other classes of operators, particularly those related to isometries \cite{curto, jung}.

Let $ B(\mathcal{H}) $ denote the algebra consisting of all bounded linear operators acting on a complex Hilbert space $ \mathcal{H} $. Throughout this work, the symbols $ \mathbb{N}, \mathbb{Z}_{+}, \mathbb{Z}, \mathbb{R}, $ and $ \mathbb{C} $ represent the sets of natural numbers, positive integers, integers, real numbers, and complex numbers, respectively. Let $ \kappa \in \mathbb{N} $, and denote by $ J_{[0,\kappa]}$ the set $ \{0, 1, 2, \ldots, \kappa\} $.

  For an operator $ T \in B(\mathcal{H}) $ and $ m \in \mathbb{Z}_{+} $, we define  
$$\mathcal{B}_{m}(T) = \sum_{j=0}^{m} (-1)^{j} \begin{pmatrix} 	m \\ j 
\end{pmatrix}T^{*(m-j)} T^{(m-j)},$$
where \( T^{*} \) denotes the adjoint of $ T $, and  $ \begin{pmatrix} m \\ j 
\end{pmatrix} $ represents the binomial coefficient.

   For $m\in\mathbb{N}$, an operator  $T\in B(\mathcal{H})$ is said to be  $m$-isometric if $ \mathcal{B}_m(T)=0$. For $k,m \in \mathbb{N}$, an operator $T \in B(H)$ is said to be   $k$-quasi-$m$-isometric if 
$ T^{*k}\mathcal{B}_m(T)T^k=0$.

Let $(X,\mathcal{F},\mu)$ be a discrete measure space, where $X$ is a countably infinite set and $\mu$ is a positive measure on $\mathcal{F}$, the $\sigma$-algebra of all subsets of $X$, such that $\mu(\{x\})>0$ for every $x\in X$.  
A function $\phi : X \to X$ is said to be measurable if $\phi^{-1}\mathcal{F} \subset \mathcal{F}$.  
The measure $\mu \circ \phi^{-1}$ on $\mathcal{F}$ is defined by
$(\mu \circ \phi^{-1})S = \mu(\phi^{-1}S), ~~~~~S \in \mathcal{F}.$

If $\mu \circ \phi^{-1}$ is absolutely continuous with respect to $\mu$, then $\phi$ is called nonsingular. In this case, by the Radon--Nikodym theorem, there exists a Radon--Nikodym derivative of $\mu \circ \phi^{-1}$ with respect to $\mu$, denoted by $h$. Moreover, if $\phi$ is nonsingular, then $\phi^{p}$ is also nonsingular for every $p \in \mathbb{Z}_{+}$. The Radon--Nikodym derivative of $\mu \circ \phi^{-p}$ with respect to $\mu$ is denoted by $h_{p}$. In particular, $h_{0}=1 $ and  $h_{1}=h.$

Let $L^{2}(X,\mathcal{F},\mu)\,(=L^{2}(\mu))$ denote the space of all equivalence classes of square integrable, complex–valued functions on $X$ with respect to the measure $\mu$. For a nonsingular measurable transformation $\phi:X\rightarrow X$, the associated composition operator $C_\phi : L^{2}(\mu) \to L^{2}(\mu)$
is defined by $C_\phi f = f \circ \phi,~~f \in L^{2}(\mu)$. The operator $C_\phi$ is bounded if and only if the Radon--Nikodym derivative $h$ is essentially bounded.  
In this case, $\|C_\phi\|^{2} = \|h\|_{\infty}$, and for each $f \in L^{2}(\mu)$ and $n \in \mathbb{Z}_{+}$,
$$
\|C_\phi^{\,n} f\|^{2} = \int_{X} h_{n}\, |f|^{2}\, d\mu.
$$

Let $L^{\infty}(\mu)$ denote the space of equivalence classes of essentially bounded, measurable, complex–valued functions on $X$.  
For $\pi \in L^{\infty}(\mu)$, the multiplication operator                              $ M_\pi : L^{2}(\mu) \to L^{2}(\mu)$ is given by
$ M_\pi f = \pi f,~ f \in L^{2}(\mu)$. If $\phi$ is a nonsingular measurable transformation on $X$ and $\pi \in L^{\infty}(\mu)$, then the weighted composition operator $ W : L^{2}(\mu) \to L^{2}(\mu)$ induced by $\phi$ and $\pi$ is defined by
$Wf = \pi \,(f \circ \phi), ~f \in L^{2}(\mu)$. For each $k \in \mathbb{Z}_{+}$, define
$\pi_{k} = \pi \, (\pi \circ \phi)\, (\pi \circ \phi^{2}) \cdots (\pi \circ \phi^{k-1})$. Then the $k$-th iterate of $W$ satisfies
$W^{k} f = \pi_{k}\, (f \circ \phi^{k}),~ f \in L^{2}(\mu)$.General properties of composition operators may be found in the standard literature.

If $\phi$ is a measurable transformation, then $\phi^{-1}\mathcal{F}$ forms a $\sigma$-subalgebra of $\mathcal{F}$, and the space $L^{2}(X,\phi^{-1}\mathcal{F},\mu)$
is a closed subspace of the Hilbert space $L^{2}(X,\mathcal{F},\mu) = L^{2}(\mu)$.  
The conditional expectation operator
$E : L^{2}(X,\mathcal{F},\mu) \rightarrow L^{2}(X,\phi^{-1}\mathcal{F},\mu)$ associated with the $\sigma$-subalgebra $\phi^{-1}\mathcal{F}$ is the orthogonal projection onto this subspace. It is defined for all non-negative measurable functions $f$ on $X$, and for all $f \in L^{2}(X,\mathcal{F},\mu)$. For each such $f$, the function $E(f)$ is the unique $\phi^{-1}\mathcal{F}$-measurable function satisfying
$$
\int_{S} f\, d\mu = \int_{S} E(f)\, d\mu, 
~~ \text{for all } S \in \phi^{-1}\mathcal{F}.
$$
We denote by $E_{n}$ the conditional expectation associated with the $\sigma$-subalgebra $\phi^{-n}\mathcal{F}$.  
If $\phi^{-n}\mathcal{F}$ is a purely atomic $\sigma$-subalgebra generated by atoms $\{A_{k}\}_{k \ge 0}$, then for every $f \in L^{2}(\mu)$,
$$
E_{n}(f \mid \phi^{-n}\mathcal{F})
= \sum_{k=0}^{\infty}
\frac{1}{\mu(A_{k})}
\left( \int_{A_{k}} f\, d\mu \right)
\chi_{A_{k}}.
$$
In this note, we discuss completion problem for the class of  $k$-quasi $m$-isometries for the classical weighted shifts, composition operators and weighted composition operators on directed graphs with one circuit and more than one branching vertex.

\section{ $k$-quasi-$m$-isometric completion problem  for weighted shift and  composition operators }  Jablonski and Kosmider studied completion problem for  m-isometric  weighted shift and composition operators on directed graphs with one circuit \cite{ZJJK } and Exner, Jung, Stochel, and Yun investigated the problem of subnormal completion in the setting of weighted shifts on directed trees \cite{g4,EJSY}. In this section we study completion problem  for weighted shift and  composition operators for $k$-quasi-$m$-isometry on more general graph settings. We begin with the following lemma by Jablonski and Kosmider \cite{ZJJK }.
 \begin{lemma}\cite{ZJJK }\label{l1}
 	Let $l \in \mathbb{Z}^+ $ and $\{ b_n \}_{n=0}^l\subset (0, \infty) $. Then, there exists $ c \in (0, \infty) $ such that for every $ t \in [c, \infty) $, there exists a polynomial $ w_t(x) \in \mathbb{R}[x]$ of degree $ l+1 $ such that:
 	$w_t(n) = b_n,~~\text{for}~~ n \in J_{[0, l]},$
 	$w_t(l+1) = t,$ and $w_t(n) > 0,~~\text{for}~~ n \in J_{[l+2, \infty]}.$
 	
\end{lemma}
\begin{proposition}\label{p1}
	Let $ m, l\in \mathbb{N}$, $k\in \mathbb{Z}^+$ and let $ \tilde{\lambda} := \{ \lambda_n \}_{n=1}^l \subset (0, \infty)$. Then the following holds:
		
	\begin{enumerate}
			
	\item If $ l \leq k+ m - 2$,   and an arbitrary sequence $ \{ \lambda_n \}_{n=l+1}^{k+m-2} \subset (0, \infty)$, then there exist a strict $k$-quasi-$ m $-isometric completion with initial weight $ \tilde{\lambda} := \{ \lambda_n \}_{n=1}^l \subset (0, \infty)$.
			
	\item If $ l > k+ m - 2 $, then there exist a $ k$-quasi-$ m $-isometric completion with initial weight  $ \tilde{\lambda} := \{ \lambda_n \}_{n=1}^l \subset (0, \infty)$ if and only if 
			\begin{equation}\label{eqn2}
				\sum_{j=0}^{m} (-1)^j\begin{pmatrix} m \\ j \end{pmatrix} (\beta{\tilde{\lambda}})_{k+n+j} = 0 \quad \text{for} \quad n \in J_{[0, ~~ l-(k+m)]},
			\end{equation} where 
		$$ (\beta \tilde{\lambda})_s= \left\{
		\begin{array}{ll}
			1&, ~~s=0\\
			\Pi_{i=1}^{s}\lambda_i^2&, ~~s\in J_{[1, \infty)} \\
		\end{array}\right.$$
	and there is a unique polynomial $ w $ of degree at most $ m - 1 $ that satisfies  $ w(n) = (\beta{\tilde{\lambda}})_{k+n} $ for $ n \in J_{[0, m-1]}$ also satisfies $w(n) > 0 $ for all $ n \in J_{[m, \infty]}$. Furthermore, this completion is strict $ k$-quasi-$ m $-isometric if and only if the degree of $w$ is exactly $m-1.$
		\end{enumerate}
			
\end{proposition}

\begin{proof}
(i) Given that $ m, l\in \mathbb{N}$, $k\in \mathbb{Z}^+$ and let $ \tilde{\lambda} := \{ \lambda_n \}_{n=1}^l \subset (0, \infty)$. Assume that $ l \leq k+ m - 2$. Choose an arbitrary sequence $ \{ \lambda_n \}_{n=l+1}^{k+m-2} \subset (0, \infty)$. Let $\tilde{\lambda}= \{ \lambda_n \}_{n=1}^{k+m-2}$ and $ \tilde{\alpha}= \{ \alpha_s \}_{s=1}^{m-2}$ such that $\alpha_s = \lambda_{k+s}, ~~\text{for} ~~ s \in J_{[1, ~~ m-2]}.$
Define $\beta$ transformation as 
$$ (\beta \tilde{\alpha})_s= \left\{
\begin{array}{ll}
	1&, ~~s=0\\
	 \Pi_{i=1}^{s}\alpha_i^2&, ~~s\in J_{[1, m-2]} \\
\end{array}\right.$$
Therefore, $\{(\beta \tilde{\alpha})_s\}_{s=0}^{m-2}$ is finite sequence of $m-1$ terms. Then by Lemma \ref{l1}, there exists $d \in (0, \infty)$ such that for every $t\in [d, \infty)$ there exists a polynomial $ w_t(x) \in \mathbb{R}[x]$ such that 
degree of $w_t(x)$is equal to $m-1$, $ w_t(s) = (\beta {\tilde{\alpha}})_{s}, ~~  s \in J_{[0, ~~m-2]}$,  $w_t(m-1)=t$ and 
	 $w_t(s) >0 ~~\text{for} ~~s \in J_{[m,~~ \infty]}.$
	Let $ \alpha_s = \sqrt{\frac{w_t(s)}{w_t(s-1)}}, ~~s \in J_{[m-1,~~ \infty]} $.
Then $ \lambda_{k+s}= \sqrt{\frac{w_t(s)}{w_t(s-1)}}, ~~s \in J_{[m-1,~~ \infty]} $. 

Let $S_\lambda$ be a unilateral weighted shift with weight sequence $\tilde{\lambda} := \{ \lambda_n \}_{n=1}^\infty $. Now 
$ (\gamma_ {S_\lambda,e_k})_{s} =  \left \| S_\lambda^{s}e_k\right \|^{2} = \Pi_{i=1}^{s}\lambda_{k+i}^2 = (\beta {\tilde{\alpha}})_{s} = w_t(s), ~~s \in J_{[0,~~ \infty)}$. But  $S_\lambda$ is $k$-quasi-$m$-isometry if and only if 
$\sum^{m}_{j=0} (-1)^{j}\begin{pmatrix} m \\
	j \end{pmatrix}\left \| S_\lambda^{s+j}e_k\right \|^{2} =0, ~~s \in J_{[0,~~ \infty)}$. Since degree of  $w_t(x)$ is $m-1$, $S_\lambda$ is a strict $k$-quasi-$m$-isometry completion of the weight sequence $\tilde{\lambda} := \{ \lambda_n \}_{n=1}^\infty $.

To prove (ii) we assume that  $ l > k+ m - 2 $ and $ \tilde{\lambda} := \{ \lambda_n \}_{n=1}^l $. Let us choose $ \tilde{\alpha}= \{ \alpha_s \}_{s=1}^{m-2}$ is a finite sequence of $m-2$ non-negative terms as in the proof of part(i). Thus, the unilateral weighted shift $S_\lambda$ is a $k$-quasi-$m$- isometry if and only if $\sum^{m}_{j=0} (-1)^{j}\begin{pmatrix} m \\
	j \end{pmatrix}(\beta {\tilde{\alpha}})_{s+j} =0, ~~s \in J_{[0,~~ \infty)}.$
Equivalently,  
there exists a unique polynomial $ w(x) \in \mathbb{R}[x]$ such that 
degree of $w(x)$ is less than or equal to  $m-1$, $ w(s) = (\beta {\tilde{\alpha}})_{s}, ~~  s \in J_{[0, ~~m-1]}$ and  $w(s) >0, \text{for} ~~s \in J_{[m,~~ \infty]}$.
	
\end{proof}

\begin{example}
Let $m=3,~~k=2,~~l=3$ and let $ \tilde{\lambda} := \{ \lambda_n \}_{n=1}^3 $ such that $\lambda_1=1, \lambda_2=3, \lambda_3=2$.
We have $l\leq k+m-2$.  
Choose $\alpha_s = \lambda_{k+s}, s\in J_{[1, m-2]}$ and so  $\{(\beta {\tilde{\alpha}})_{s}\}_{s=0}^{m-2} = \{1, \alpha_1^2\}= \{1, 4\}.$ Then by Proposition\ref{p1}(i),there exists a polynomial $ w(x) \in \mathbb{R}[x]$ such that 
	 degree of $w(x)$is equal to $m-1= 2$, $ w(s) = (\beta {\tilde{\alpha}})_{s}, ~~  s \in J_{[0, ~~m-2]}$,  $w(s) >0, \text{for} ~~s \in J_{[m-1,~~ \infty]}.$
Therefore, $w(x)=ax^2+bx+c$ ~~such that ~~$a\ne0, c= 1, a+b=3$ and $w(s)>0, \text{for} ~~s \in J_{[m-1,~~ \infty]}$. In particular, if $a=3, b=0,$ and $c=1$ implies $w(x)=3x^2+1$. In this case we can take $\alpha_s = \sqrt{\frac{w(s)}{w(s-1)}}, \text{for} ~~s \in J_{[m-1,~~ \infty]}$. 
Hence, $  \tilde{\lambda} := \{1, 3, 2,\sqrt{\frac{13}{4}}, \sqrt{\frac{28}{13}} \ldots \}$ admits a strict $2$-quasi-$3$-isomeric completion with intial weights 
$  \tilde{\lambda} := \{1, 3, 2 \}$ 
\end{example}

\begin{example}\label{eg1}
Let $m=4,~~k=2,~~l=5$ and let $ \tilde{\lambda} := \{ \lambda_n \}_{n=1}^5 $ such that $\lambda_1=2, \lambda_2=5, \lambda_3=3, \lambda_4=1, \lambda_5=2$. Note that 
$l> k+m-2$. 
If we take   $\alpha_s = \lambda_{k+s}, s\in J_{[1, m-2]}$, $\{(\beta {\tilde{\alpha}})_{s}\}_{s=0}^{m-2} = \{1, 9, 9\}$. Since $l-(k+m)=-1$, $J_{[0, l-(k+m)]}=\emptyset$. Therefore \ref{eqn2} is trivially true. Now,  to verify the existence of $k$-quasi-$m$ -isometric completion with initial weights $ \tilde{\lambda} := \{ \lambda_n \}_{n=1}^5 $, we have to find a unique polynomial  $ w(x) \in \mathbb{R}[x]$ such that 
	 degree of $w(x)$is less than or equal to $m-1= 3$,
	 $ w(s) = (\beta {\tilde{\alpha}})_{s}, ~~  s \in J_{[0, ~~m-2]}$,
	$w(s) >0, \text{for} ~~s \in J_{[m-1,~~ \infty]}$.
Clealy, it is impossible to take degree of  $w(x)$ is equal to $0,1$ or $2$. So, there exist a unique polynomial $w(x)= \frac{35}{6}x^3-\frac{43}{2}x^2+\frac{71}{3}x+1$ satifies the required condition. Hence there exist a strict $2$-quasi-$4$-isometric completion with weights  $\tilde{\lambda} := \{ \lambda_n \}_{n=1}^ \infty $, where $\lambda_{k+n} = \sqrt{\frac{w(n)}{w(n-1)}}, n \in J_{[1,~~ \infty]}.$
\end{example}
\begin{remark}
In  Example \ref{eg1}, if we change the value of $m= 3$, then there does not exist a $2$-quasi-$3$-isometric completion with weights  $\tilde{\lambda} := \{ \lambda_n \}_{n=1}^ \infty $.
\end{remark}
 

Now we study $k$-quasi-$m$-isometric  completion problem for composition operators on directed graphs. We start with the basic structure of directed graph setting that we considered in \cite{DSP}


Let $\kappa\in \mathbb{N}$, $\eta_r\in\mathbb{Z_+\cup\{\infty\}}$, for               $r\in J_{[1,\kappa]}$ and atleast one of $\eta_r$ is nonzero. Throughout this  section we let
\begin{equation}\label{eqn1}
	 X= X_\kappa \cup\bigcup_{r=1}^{\kappa}\bigcup_{i=1}^{\eta_r} X_{\eta_r},
\end{equation}
where $X_\kappa=\{x_1,x_2,\ldots,x_k\} ~~\text{and}~~
X_{\eta_r}=\bigcup_{i=1}^{\eta_r} \{x^r_{i,j}: j\in\mathbb{N}\},~~r\in J_{[1,\kappa]}$
be two disjoint collections of distinct points in $X$.  

Consider $X$ as a directed graph containing a single circuit, where $\{x_1,x_2,\ldots,x_k\}$ forms the set of branching vertices on this circuit.  
For each $r\in J_{[1,\kappa]}$, the set $X_{\eta_r}$ represents the branching elements associated with $x_r$, and  $\{x^r_{i,j}: j\in\mathbb{N}\}$ denotes the vertices along the $i^{\text{th}}$ branch emanating from $x_r$ for $i\in J_{[1,\eta_r]}$.  
Here, $\eta_r$ denotes the number of branches originating from the vertex $x_r$. Recently, a more general form of this class of graphs has been applied by Buchala \cite{MB} in the investigation of $m$-isometric composition operators on discrete spaces, as well as in the study of the subnormality problem for Cauchy dual operators. Figure 1 illustrates this construction for the specific case $\kappa = 4$ and $\eta_{r}=2$ for all $r\in J_{[1,\kappa]}$.
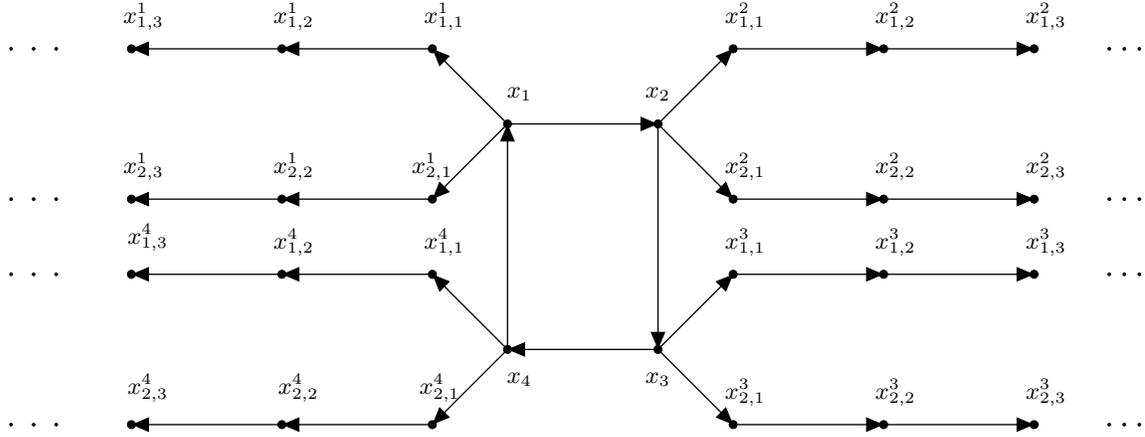
\begin{figure}
	\centering
	\begin{tikzpicture}[line cap=round,line join=round,>=triangle 45,x=1cm,y=1cm]
		\clip(-8,-3.28) rectangle (13.72,5.24);
		\draw [->,line width=.5pt] (-1,2) -- (1,2);
		\draw [->,line width=.5pt] (1,2) -- (1,-1);
		\draw [->,line width=.5pt] (1,-1) -- (-1,-1);
		\draw [->,line width=.5pt] (-1,-1) -- (-1,2);
		\draw [->,line width=.5pt] (1,2) -- (2,3);
		\draw [->,line width=.5pt] (1,2) -- (2,1);
		\draw [->,line width=.5pt] (2,1) -- (4,1);
		\draw [->,line width=.5pt] (4,1) -- (6,1);
		\draw [->,line width=.5pt] (1,-1) -- (2,0);
		\draw [->,line width=.5pt] (2,0) -- (4,0);
		\draw [->,line width=.5pt] (4,0) -- (6,0);
		\draw [->,line width=.5pt] (-1,2) -- (-2,3);
		\draw [->,line width=.5pt] (-2,3) -- (-4,3);
		\draw [->,line width=.5pt] (-4,3) -- (-6,3);
		\draw [->,line width=.5pt] (-1,2) -- (-2,1);
		\draw [->,line width=.5pt] (-2,1) -- (-4,1);
		\draw [->,line width=.5pt] (-4,1) -- (-6,1);
		\draw [->,line width=.5pt] (2,3) -- (4,3);
		\draw [->,line width=.5pt] (4,3) -- (6,3);
		\draw [->,line width=.5pt] (1,-1) -- (2,-2);
		\draw [->,line width=.5pt] (2,-2) -- (4,-2);
		\draw [->,line width=.5pt] (4,-2) -- (6,-2);
		\draw [->,line width=.5pt] (-1,-1) -- (-2,0);
		\draw [->,line width=.5pt] (-2,0) -- (-4,0);
		\draw [->,line width=.5pt] (-4,0) -- (-6,0);
		\draw [->,line width=.5pt] (-1,-1) -- (-2,-2);
		\draw [->,line width=.5pt] (-2,-2) -- (-4,-2);
		\draw [->,line width=.5pt] (-4,-2) -- (-6,-2);
		\begin{scriptsize}
			\draw [fill=black] (-1,2) circle (1.5pt);
			\draw[color=black] (-0.84,2.4) node {$x_1$};
			\draw [fill=black] (1,2) circle (1.5pt);
			\draw[color=black] (1,2.4) node {$x_2$};
			\draw [fill=black] (1,-1) circle (1.5pt);
			\draw[color=black] (1,-1.4) node {$x_3$};
			\draw [fill=black] (-1,-1) circle (1.5pt);
			\draw[color=black] (-0.84,-1.4) node {$x_4$};
			\draw [fill=black] (2,3) circle (1.5pt);
			\draw[color=black] (2.16,3.43) node {$x_{1,1}^2$};
			\draw [fill=black] (2,1) circle (1.5pt);
			\draw[color=black] (2.16,1.43) node {$x_{2,1}^2$};
			\draw [fill=black] (4,3) circle (1.5pt);
			\draw[color=black] (4.16,3.43) node {$x_{1,2}^2$};
			\draw [fill=black] (6,3) circle (1.5pt);
			\draw[color=black] (6.16,3.43) node {$x_{1,3}^2$};
			\draw [fill=black] (7,3) circle (.5pt);
			\draw [fill=black] (7.2,3) circle (.5pt);
			\draw [fill=black] (7.4,3) circle (.5pt);
			\draw [fill=black] (7,1) circle (.5pt);
			\draw [fill=black] (7.2,1) circle (.5pt);
			\draw [fill=black] (7.4,1) circle (.5pt);
			\draw [fill=black] (7,0) circle (.5pt);
			\draw [fill=black] (7.2,0) circle (.5pt);
			\draw [fill=black] (7.4,0) circle (.5pt);
			\draw [fill=black] (7,-2) circle (.5pt);
			\draw [fill=black] (7.2,-2) circle (.5pt);
			\draw [fill=black] (7.4,-2) circle (.5pt);
			\draw [fill=black] (-7,3) circle (.5pt);
			\draw [fill=black] (-7.3,3) circle (.5pt);
			\draw [fill=black] (-7.6,3) circle (.5pt);
			\draw [fill=black] (-7,1) circle (.5pt);
			\draw [fill=black] (-7.3,1) circle (.5pt);
			\draw [fill=black] (-7.6,1) circle (.5pt);
			\draw [fill=black] (-7,0) circle (.5pt);
			\draw [fill=black] (-7.3,0) circle (.5pt);
			\draw [fill=black] (-7.6,0) circle (.5pt);
			\draw [fill=black] (-7,-2) circle (.5pt);
			\draw [fill=black] (-7.3,-2) circle (.5pt);
			\draw [fill=black] (-7.6,-2) circle (.5pt);
			
			\draw [fill=black] (4,1) circle (1.5pt);
			\draw[color=black] (4.16,1.43) node {$x_{2,2}^2$};
			\draw [fill=black] (6,1) circle (1.5pt);
			\draw[color=black] (6.16,1.43) node {$x_{2,3}^2$};
			\draw [fill=black] (2,0) circle (1.5pt);
			\draw[color=black] (2.16,0.43) node {$x_{1,1}^3$};
			\draw [fill=black] (4,0) circle (1.5pt);
			\draw[color=black] (4.16,0.43) node {$x_{1,2}^3$};
			\draw [fill=black] (6,0) circle (1.5pt);
			\draw[color=black] (6.16,0.43) node {$x_{1,3}^3$};
			\draw [fill=black] (-2,3) circle (1.5pt);
			\draw[color=black] (-1.84,3.43) node {$x_{1,1}^1$};
			\draw [fill=black] (-4,3) circle (1.5pt);
			\draw[color=black] (-3.84,3.43) node {$x_{1,2}^1$};
			\draw [fill=black] (-6,3) circle (1.5pt);
			\draw[color=black] (-5.84,3.43) node {$x_{1,3}^1$};
			\draw [fill=black] (-2,1) circle (1.5pt);
			\draw[color=black] (-2,1.43) node {$x_{2,1}^1$};
			\draw [fill=black] (-4,1) circle (1.5pt);
			\draw[color=black] (-3.84,1.43) node {$x_{2,2}^1$};
			\draw [fill=black] (-6,1) circle (1.5pt);
			\draw[color=black] (-5.84,1.43) node {$x_{2,3}^1$};
			
			\draw [fill=black] (2,-2) circle (1.5pt);
			\draw[color=black] (2.16,-1.57) node {$x_{2,1}^3$};
			\draw [fill=black] (4,-2) circle (1.5pt);
			\draw[color=black] (4.16,-1.57) node {$x_{2,2}^3$};
			\draw [fill=black] (6,-2) circle (1.5pt);
			\draw[color=black] (6.16,-1.57) node {$x_{2,3}^3$};
			\draw [fill=black] (-2,0) circle (1.5pt);
			\draw[color=black] (-1.84,0.43) node {$x_{1,1}^4$};
			\draw [fill=black] (-4,0) circle (1.5pt);
			\draw[color=black] (-3.84,0.43) node {$x_{1,2}^4$};
			\draw [fill=black] (-6,0) circle (1.5pt);
			\draw[color=black] (-5.78,0.49) node {$x_{1,3}^4$};
			\draw [fill=black] (-2,-2) circle (1.5pt);
			\draw[color=black] (-1.9,-1.51) node {$x_{2,1}^4$};
			\draw [fill=black] (-4,-2) circle (1.5pt);
			\draw[color=black] (-3.78,-1.51) node {$x_{2,2}^4$};
			\draw [fill=black] (-6,-2) circle (1.5pt);
			\draw[color=black] (-5.78,-1.51) node {$x_{2,3}^4$};
		\end{scriptsize}
	\end{tikzpicture}
\caption{Directed graph with one circuit and more than one branching vertex}
\end{figure}\label{fig:Directed graph with one circuit}

Let $(X, \mathcal{F}, \mu)$ be a $\sigma$-finite measure space, where $\mu$ is a $\sigma$-finite positive measure on $X$ such that $\mu(\{x\})>0$, for all $x\in X$.  
To describe the parent function on $(X,\mathcal{F},\mu)$, which will be used to characterize the atoms of the $\sigma$-algebra $\phi^{-p}(\mathcal{F})$ within $\mathcal{F}$, we introduce two auxiliary functions $\Phi_1$ and $\Phi_2$.  

Fix $\kappa\in\mathbb{N}$, and define $\Phi_1:\mathbb{Z}\to\mathbb{Z}$ and $\Phi_2:\mathbb{Z}\to J_{[1,\kappa]}$ uniquely by the relation
$$
p=\Phi_1(p)\kappa+\Phi_2(p),\qquad p\in\mathbb{Z}.
$$
These functions satisfy
$$
\Phi_1(l\kappa+1)=\Phi_1(l\kappa+r),\qquad l\in\mathbb{Z},~ r\in J_{[1,\kappa]},
$$
$$
\Phi_2(l\kappa+r_1+r_2)=\Phi_2(l\kappa+r_1)+r_2,~~ l\in\mathbb{Z}, ~~~\text{for}~~~ r_1\in\mathbb{N},~ r_2\in\mathbb{Z}_+,~ r_1+r_2\in J_{[1,\kappa]}.  
$$

Using the directed graph described above, we now obtain the corresponding parent function as follows:

\begin{align*}
	par(x)= \left\{
	\begin{array}{ll}\
		x^r_{i,j}, & \mathrm{if}~~ x=x^r_{i,j+1} ~~~\mathrm{for}~~r\in J_{[1,\kappa]},~~~ i\in J_{[1,\eta_r]},~~ \mathrm{and}~~j\in \mathbb{N} ,\\\\
		x_r, & \mathrm{if}~~ x=x^s_{i,j}, ~~~\mathrm{for}~~ s\in J_{[1,\kappa]}  \mathrm{~and}~ \Phi_2(1+r)=\Phi_2(s+j),~  j\in\mathbb{N}, \\
		& i\in J_{[1,\eta_s]},~~ \mathrm{or}~~x=x_{\Phi_2(1+r)}.
	\end{array}\right. \end{align*}

\begin{align}\label{equ1} 
\textrm{ Assume that } (X, \mathcal{F}, \mu) \textrm{ is a discrete mesaure space  and } \phi : X\rightarrow X \textrm{ is a}\nonumber \\~~   \textrm{ nonsingular measurable transformation on X ~~defined by }
	\phi(x)=par(x),~~x\in X.
\end{align}


Given that $\mu(x) > 0$ for every $x \in X$, the transformation $\phi$ is nonsingular, and consequently, $\phi^p$ is also nonsingular for $p \in \mathbb{N}$. Therefore, the Radon-Nikodym derivative $h_p = \frac{d(\mu \circ \phi^{-p})}{d\mu}$ can be determined using the atoms of the $\sigma$-algebra $\phi^{-p}(\mathcal{F})$ (see \cite{DSP}) as follows:
$$h_p(x)= \left\{
\begin{array}{ll}
	\frac{\mu(x^r_{i,j+p})}{\mu(x^r_{i, j})}, & \mathrm{if}~~ x=x^r_{i,j},  ~~ r\in J_{[1,\kappa]},  ~~  i\in J_{[1,\eta_r]}, \\
	& ~~j\in \mathbb{N},\\\\
	\frac{\mu(x_{\Phi_{2}(p+r)})
		+ \displaystyle\sum_{j=1}^{p}
		\sum_{\substack{s=1\\\Phi_{2}(p+r)=\Phi_{2}(s+j)}}^{\kappa}
		\sum_{i=1}^{\eta_{s}}
		\mu(x^{s}_{i,j})}{\mu(\{x_r\})},	 &\mathrm{if}~~ x=x_r,~~r\in J_{[1,\kappa]} .\\\\
\end{array}\right.$$


 \begin{theorem}\label{Tm1}
 Let $m,\kappa,k \in \mathbb{N}$ such that $\kappa>m\geq 2$. Assume that  $\eta_{r},  X, \mathcal{F},\text{and} ~~~\phi $ are as in \eqref{equ1}, $ M\in (0, \infty)$, $\{w_{i}^{r}\}_{i=1}^{\eta_r}$ is system of polinomials of degree atmost~~~~ $m-2,$ $\{ a^r_{i,j}\}_{j=1}^\infty\subset (0, \infty)$ such that $w_{i}^{r}(j)= a^r_{i,k+j}, \textrm{for~~all~~}  r \in J_{[1,\kappa]}, i \in J_{[1,\eta_r]}, j \in \mathbb{N}$ and
  \begin{equation}\label{eqn3}
  	\underset{r \in J_{[1,\kappa]}}{\max} \left\{ \sum_{i=1}^{\eta_r} a_{i, 1}^r, \sup\{\frac{a_{i, j+1}^r}{a_{i, j}^r}:i \in J_{[1,\eta_r]}, j \in \mathbb{N}\} \right\}\leq M. 
  \end{equation}
  Then there exist a measure $\mu$ on $\mathcal{F}$ sucht that 
 \begin{enumerate}
  \item $ \mu(x^{r}_{i, j})= a_{i, j}^r, \textrm{for~~all~~} r \in J_{[1,\kappa]}, i \in J_{[1,\eta_r]} \text{and} ~~ j \in \mathbb{N} $ and \\
   \item the corresponding $C_\phi\in B({L^{2}(\mu)})$ is $k$-quasi-$m$ -isometry. 
 \end{enumerate}
Furthermore, there exists  $ c\in \mathbb{R}$ such that for every $ t\in (c, \infty)$ there is a measure $\mu_t$ with the above properties can be written as the form  
$$\mu_t(x) =
\begin{cases}
	\mu(x) + t, & \text{if } x = x_r \text{ for} ~~r \in J_\kappa, \\
	\mu(x), & \text{otherwise}.
\end{cases} $$
 \end{theorem}
\begin{proof}
Define $ \mu(x^{r}_{i, j})=a_{i, j}^r, \text{ for~ all}~~r \in J_{[1,\kappa]}, i \in J_{[1,\eta_r]} , ~~ j \in \mathbb{N} $. Then by (\ref{eqn3}) the Radon-Nikodym derivative $h$ is essentially bounded and so $C_\phi \in B(L^2(\mu))$. By  \cite[Theorem 2.6]{DSP}  $C_\phi $ is $k$-quasi-$m$ -isometry if and only if 
$ \{\mu(x^{r}_{i,k+j+1})\}_{j=0}^\infty $ is a polynomial in $j$ of degree at most $m-2$ for every $r \in J_{[1,\kappa]}, i \in J_{[1,\eta_r]} $ and  
$	\displaystyle\sum_{p=0}^{m} (-1)^p \begin{pmatrix} m \\
	p \end{pmatrix} h_{p+k}(x_r)= 0,  ~~ \textrm{for all}~~r \in J_{[1,\kappa]} .$\\

From the hypothesis it is clear that $ \{\mu(x^{r}_{i,k+j+1})\}_{j=0}^\infty $ is a polynomial in $j$ of degree at most $m-2$ for every $r\in J_{[1,\kappa]} $, $i\in J_{[1,\eta_r]}$.
Now   
$	\displaystyle\sum_{p=0}^{m} (-1)^p \begin{pmatrix} m \\
	p \end{pmatrix} h_{p+k}(x_r)= 0,  ~~ \textrm{for all}~~r \in J_{[1,\kappa]}  $ if and only if \\
$\displaystyle\sum_{p=0}^{m} (-1)^p \begin{pmatrix} m \\
	p \end{pmatrix}(\mu(x_{{\Phi_2}{(p+k+r)}})+\displaystyle\sum_{j=1}^{p+k}\displaystyle\sum_{\substack {s=1,   \\\Phi_2(p+k+r)=\Phi_2(s+j)} }^{\kappa}
\displaystyle\sum _{i=1}^{\eta_{s}} \mu(x^s_{i,j}))= 0, \textrm{~~for all}~~r \in J_{[1,\kappa]} . $
That is,
 \begin{equation}\label{eqn4}
 \displaystyle\sum_{p=0}^{m} (-1)^p \begin{pmatrix} m \\
 		p \end{pmatrix}\mu(x_{{\Phi_2}{(p+k+r)}})=b_r, \textrm{for all}~~r \in J_{[1,\kappa]} , 
 \end{equation}
 where $ b_r= \displaystyle\sum_{p=0}^{m} (-1)^{p+1} \begin{pmatrix} m \\
 	p \end{pmatrix}\S\displaystyle\sum_{j=1}^{p+k}\displaystyle\sum_{\substack {s=1,   \\\Phi_2(p+k+r)=\Phi_2(s+j)} }^{\kappa}
 \displaystyle\sum _{i=1}^{\eta_{s}} \mu(x^s_{i,j}))= 0, \textrm{~~for all}~~r \in J_{[1,\kappa]} . $\\
 Let 
 $$a_p :=
 \begin{cases}
 	(-1)^p \begin{pmatrix} m \\
 		p \end{pmatrix} &, \text{if } 0\leq p \leq m, \\
 	0 &, \text{if } m<p \leq{\kappa-1}
 \end{cases}. $$
Then (\ref{eqn4}) can be written as follows
\begin{equation}\label{eqn5}
	 AX=B
\end{equation}
where $ A = \begin{bmatrix}
	a_0 & a_1 & a_2 & \ldots& a_{\kappa-2} & a_{\kappa-1}\\
	 a_{\kappa-1} & a_0 & a_1 & \ldots & a_{\kappa-3} & a_{\kappa-2} \\
	 a_{\kappa-2} &  a_{\kappa-1} & a_0 &  \ldots & a_{\kappa-4} & a_{\kappa-3}\\
	 \vdots & \vdots& \vdots& \ldots & \vdots & \vdots\\
	 a_1 & a_2 & a_3 & \ldots & a_{\kappa-1} & a_0 
\end{bmatrix} $, 
$ X = \begin{bmatrix}
	\mu(x_{{\Phi_2}{(k+1)}})\\
	\mu(x_{{\Phi_2}{(k+2)}})\\
	\vdots\\
	\mu(x_{{\Phi_2}{(k+\kappa)}})
	\end{bmatrix}$, and 
$ B= \begin{bmatrix}
b_1\\
	b_2\\
	\vdots\\
	b_\kappa
\end{bmatrix}$

 Since $ \mu(x_{{\Phi_2}{(k+r)}}), r \in J_{[1,\kappa]}$ are the rearrangement of $ \mu(x_1), \mu(x_2),\ldots \mu(x_\kappa)$, then we have  $ rank(A) = \kappa-1 = rank(AB)$. Therefore, by applying the Rouché–Capelli theorem,  there exist a solution for (\ref{eqn5}). 
 
  To establish the second part, observe that there exist a positive real  number $c$ such that $\mu(x) +t >0$, for all $ t \in [c, \infty)$ and the vector $ v = [t, t, \ldots, t] \in \mathbb{R}^{\kappa} $ satisfies $ Av = 0 $. Consequently, for every solution $ X $ of (\ref{eqn5}), the relation $ A(X + v) = B $ also holds. 
  
\end{proof}
  
 We now study a complete characterization of $1$-quasi-$3$-isometric and
 $1$-quasi-$2$-isometric composition operators acting on above type of directed graphs.
 Throughout this section, we restrict attention to the case where the length of
 the circuit satisfies $1 < \kappa \leq 3$.
  \begin{theorem}\label{Tm2}
  	Let $\kappa=2$.  Assume that  $\eta_{r},  X, \mathcal{F},\text{and} ~~~\phi $ are as in \eqref{equ1}. Then $C_\phi\in B({L^{2}(\mu)})$ is $1$-quasi-$3$-isometry if and only if there exist $M, t\in(0,\infty)$ and a system of polynomials $\{ w^r_{i,j}\}_{i=1}^{\eta_{r}}$ of degree atmost 1 and two systems $ \{ c^r_i\}_{i=1}^{\eta_{r}} \subset (0,\infty)$ and $ \{d^r_i\}_{i=1}^{\eta_{r}} \subset \mathbb{R}_{+} $ such that 
  	\begin{equation}\label{eqn6}
  		\underset{r \in J_{[1,\kappa]} }{\max} \left\{ \sum_{i=1}^{\eta_r} \mu (x_{i, 1}^r), \sup\{\frac{\mu (x_{i, j+1}^r)}{\mu (x_{i, j}^r)}:i \in J_{[1,\eta_r]} , j \in \mathbb{N}\} \right\}\leq M 
  	\end{equation}
  	
  and $c^{(r)}= \sum_{i=1}^{\eta_{r}}c^r_i < \infty$, $d^{(r)}= \sum_{i=1}^{\eta_{r}}d^r_i < \infty$, 
  \begin{equation}\label{eqn7}
  	 \mu (x_{i, j+1}^r)= w^r_{i,j+1}= c^r_i +d^r_i(j-1), ~~ r\in J_{[1,\kappa]} , i\in J_{[1,\eta_r]}  , j\in \mathbb{N},
  \end{equation}
  
  $$ \sum_{i=1}^{\eta_1}\mu (x_{i, 1}^1)+ \sum_{i=1}^{\eta_2}\mu (x_{i, 1}^2)= \sum_{r=1}^{2}c^{(r)},$$
  $$\mu(x_r)= W(r), r\in J_{[1,\kappa]} , $$
  where $W(x)=ax^2+bx+t$ is polynomial of degree atmost 2 such that 
  $$3a+b =\frac{7\sum_{i=1}^{\eta_1} \mu(x^1_{i,1})-5c^{(1)}+d^{(1)}-2c^{(2)}-d^{(2)}}{2\kappa}$$
  	 
  \end{theorem}
  
  \begin{proof}
  	Let $\kappa=2$. The boundedness of $C_\phi$ follows from the essential boundedness of the Radon–Nikodym derivative $h$, which holds due to the existence of a constant $M>0$ satisfying equation \ref{eqn6}.
  	Then $C_\phi\in B({L^{2}(\mu)})$ is $1$-quasi-$3$-isometry  if and only if 
  	$ \{\mu(x^{r}_{i,j+2})\}_{j=0}^\infty $ is a polynomial in $j$ of degree at most 1, for all $r\in J_{[1,\kappa]}  $, $i\in J_{[1,\eta_r]} $ and  
  	\begin{equation}\label{eqn8}
  			\sum_{p=0}^{3} (-1)^p \begin{pmatrix} m \\
  			p \end{pmatrix} h_{p+1}(x_r)= 0,  ~~ \textrm{for all}~~r \in J_{[1,\kappa]}  .
  	\end{equation}
  	That is, there is a system of polynomials $\{ w^r_{i,j}\}_{i=1}^{\eta_{r}}, ~~r\in J_{[1,\kappa]} $ of degree atmost 1  such that 
  	 $$ \mu (x_{i, j+1}^r)= w^r_{i,j+1}= c^r_i +d^r_i(j-1), ~~ r\in J_{[1,\kappa]} , i\in J_{[1,\eta_r]}  , j\in \mathbb{N},$$
  	where $ \{ c^r_i\}_{i=1}^{\eta_{r}} \subset (0,\infty)$,                  $ \{d^r_i\}_{i=1}^{\eta_{r}} \subset \mathbb{R}_{+} $ and  
  	$c^{(r)}= \sum_{i=1}^{\eta_{r}}c^r_i < \infty, ~~d^{(r)}= \sum_{i=1}^{\eta_{r}}d^r_i < \infty.$ and 
  	\begin{align*}
  	-4\mu(x_1)+4\mu(x_2)+3\sum_{i=1}^{\eta_1}\mu(x^1_{i,1})-4\sum_{i=1}^{\eta_1}\mu(x^1_{i,2})+3\sum_{i=1}^{\eta_1}\mu(x^1_{i,3})-\sum_{i=1}^{\eta_1}\mu(x^1_{i,4})\\
  	-4\sum_{i=1}^{\eta_2}\mu(x^2_{i,1})+3\sum_{i=1}^{\eta_2}\mu(x^2_{i,2})-\sum_{i=1}^{\eta_2}\mu(x^2_{i,3}) = 0,
  	\end{align*}
 
  	\begin{align*}
  		4\mu(x_1)-4\mu(x_2)-4\sum_{i=1}^{\eta_1}\mu(x^1_{i,1})+3\sum_{i=1}^{\eta_1}\mu(x^1_{i,2})-\sum_{i=1}^{\eta_1}\mu(x^1_{i,3})+3\sum_{i=1}^{\eta_2}\mu(x^2_{i,1})\\
  		-4\sum_{i=1}^{\eta_2}\mu(x^2_{i,2})+3\sum_{i=1}^{\eta_2}\mu(x^2_{i,3})-\sum_{i=1}^{\eta_2}\mu(x^2_{i,4}) = 0.
  	\end{align*}  
By solving the above equations using the polynomial representation of 
$\{\mu(x^{r}_{i,j+2})\}_{j=0}^{\infty}$, we obtain,

$C_\phi\in B({L^{2}(\mu)})$ is $1$-quasi-$3$-isometry  if and only if
 $$ \sum_{i=1}^{\eta_1}\mu (x_{i, 1}^1)+ \sum_{i=1}^{\eta_2}\mu (x_{i, 1}^2)= \sum_{r=1}^{2}c^{(r)},$$ and
$$\mu(x_r)= W(r), r\in J_{[1, \kappa]}, $$
where $W(x)=ax^2+bx+t$ is polynomial of degree at most two  such that
$$3a+b =\frac{7\sum_{i=1}^{\eta_1} \mu(x^1_{i,1})-5c^{(1)}+d^{(1)}-2c^{(2)}-d^{(2)}}{2\kappa}.$$ 

  \end{proof}
  \begin{theorem}\label{Tm3}
  Let $\kappa = 3$. Assume that $\eta_{r}$, $X$, $\mathcal{F}$, and $\phi$ are defined as in \eqref{equ1}. Then, the operator $C_\phi \in B(L^{2}(\mu))$ is a $1$-quasi-$3$-isometry if and only if there exist $M, t \in (0, \infty)$, along with a family of polynomials$\{ w^r_{i,j}\}_{i=1}^{\eta_{r}}$ of degree at most one, and two corresponding sequences $ \{ c^r_i\}_{i=1}^{\eta_{r}} \subset (0,\infty)$ and $ \{d^r_i\}_{i=1}^{\eta_{r}} \subset \mathbb{R}_{+} $ , that together satisfy conditions \ref{eqn6} and \ref{eqn7} for $\kappa = 3$ and
   $$\mu(x_r)= W(r)>0, r\in J_{[1, \kappa]}, $$
  where $W(x)=ax^2+bx+t, ~t>0$ is polynomial of degree atmost 2 such that \\
  
  $ a= \frac{6\sum_{i=1}^{\eta_1} \mu(x^1_{i,1})-3\sum_{i=1}^{\eta_2} \mu(x^2_{i,1})-3\sum_{i=1}^{\eta_3} \mu(x^3_{i,1})-3c^{(1)}+d^{(1)}+d^{(2)}+3c^{(3)}-2d^{(3)}}{2\kappa}, $
  
$ b= \frac{-24\sum_{i=1}^{\eta_1} \mu(x^1_{i,1})+9\sum_{i=1}^{\eta_2} \mu(x^2_{i,1})+15\sum_{i=1}^{\eta_3} \mu(x^3_{i,1})+11c^{(1)}-3d^{(1)}
  		+2c^{(2)}-5d^{(2)}-13c^{(3)}+8d^{(3)}}{2\kappa}. $
  
  \end{theorem}
  \begin{proof}
 Let $\kappa=3$. Following the reasoning in the proof of Theorem \ref*{Tm2} it can be concluded that $C_\phi \in B(L^{2}(\mu))$. Moreover, the operator $C_\phi$ is a $1$-quasi-$3$-isometry if and only if the following relation holds:
    $$ \mu (x_{i, j+1}^r)= w^r_{i,j+1}= c^r_i +d^r_i(j-1), ~~ r\in J_{[1, \kappa]}, i\in J_{[1,\eta_{r}]} , j\in \mathbb{N},$$
where the coefficient sets $ \{ c^r_i\}_{i=1}^{\eta_{r}} \subset (0,\infty)$,                 
 $ \{d^r_i\}_{i=1}^{\eta_{r}} \subset \mathbb{R}_{+} $ satisfy the summability conditions 
 $$c^{(r)}= \sum_{i=1}^{\eta_{r}}c^r_i < \infty, ~~d^{(r)}= \sum_{i=1}^{\eta_{r}}d^r_i < \infty.$$
Furthermore, the following identity must also be satisfied:
 	\begin{equation}\label{eqn9}
 	\sum_{p=0}^{3} (-1)^p \begin{pmatrix} m \\
 		p \end{pmatrix} h_{p+1}(x_r)= 0,  ~~ \textrm{for all}~~r \in J_{[1, \kappa]} .
 \end{equation}
Then  \ref{eqn9} becomes a sysytem of homogeneous linear equations of the form 
\begin{align*}
	3\mu(x_1)-3\mu(x_3)-3\sum_{i=1}^{\eta_2} \mu(x^2_{i,1})+3\sum_{i=1}^{\eta_3} \mu(x^3_{i,1})-c^{(1)}+d^{(1)}
	+2c^{(2)}-d^{(2)}-c^{(3)}=0\\
	-3\mu(x_1)+3\mu(x_2)+3\sum_{i=1}^{\eta_1} \mu(x^1_{i,1})-3\sum_{i=1}^{\eta_3} \mu(x^3_{i,1})-c^{(1)}
	-c^{(2)}+d^{(2)}+2c^{(3)}-d^{(3)}=0\\
	-3\mu(x_2)+3\mu(x_3)-3\sum_{i=1}^{\eta_1} \mu(x^1_{i,1})+3\sum_{i=1}^{\eta_2} \mu(x^2_{i,1})+2c^{(1)}-d^{(1)}
	-c^{(2)}-c^{(3)}+d^{(3)}=0	
\end{align*}
By using  the above system of equations we obtain $C_\phi$ is 1-quasi-3-isometry if and only if there exist $t>0$ and  a polynomial $W(x)=ax^2+bx+t$ of degree atmost 2 such that $\mu(x_r)=W(r)>0, r\in J_{[1, \kappa]}$, where \\

 $ a= \frac{6\sum_{i=1}^{\eta_1} \mu(x^1_{i,1})-3\sum_{i=1}^{\eta_2} \mu(x^2_{i,1})-3\sum_{i=1}^{\eta_3} \mu(x^3_{i,1})-3c^{(1)}+d^{(1)}+d^{(2)}+3c^{(3)}-2d^{(3)}}{2\kappa}, $\\\\
 and\\

$ b= \frac{-24\sum_{i=1}^{\eta_1} \mu(x^1_{i,1})+9\sum_{i=1}^{\eta_2} \mu(x^2_{i,1})+15\sum_{i=1}^{\eta_3} \mu(x^3_{i,1})+11c^{(1)}-3d^{(1)}
	+2c^{(2)}-5d^{(2)}-13c^{(3)}+8d^{(3)}}{2\kappa}. $

  \end{proof}
  
  \begin{theorem}
  
  Let $\kappa = 4$, and assume that the parameters $\eta_r$, the measurable space
  $(X,\mathcal{F})$, and the mapping $\phi$ are defined as in \eqref{equ1}. The
  composition operator $C_\phi \in B(L^2(\mu))$ is a $1$-quasi-$3$-isometry if and only
  if there exist constants $M,t \in (0,\infty)$ such that, for each
  $r \in J_{[1,\kappa]}$, there is a family of polynomials $\{w_i^r\}_{i=1}^{\eta_r}$ of
  degree at most one, together with sequences
  $\{c_i^r\}_{i=1}^{\eta_r} \subset (0,\infty)$ and
  $\{d_i^r\}_{i=1}^{\eta_r} \subset \mathbb{R}_+$ satisfy conditions \ref{eqn6} and \ref{eqn7}
  corresponding to the case $\kappa = 4$ and
  
  \begin{equation*}
  	-3\sum_{i=1}^{\eta_1}\mu (x_{i, 1}^1)+ 3\sum_{i=1}^{\eta_2}\mu (x_{i, 1}^2)-\sum_{i=1}^{\eta_3}\mu (x_{i, 1}^3)+\sum_{i=1}^{\eta_4}\mu (x_{i, 1}^4)= -2c^{(1)}+d^{(1)}+c^{(2)}+c^{(4)}-d^{(4)}
  \end{equation*}
 \begin{equation*}
 	\mu(x_r)= W(r)>0, r\in J_{[1, \kappa]}, 
 \end{equation*}
  Where $W(x)=ax^2+bx+t, ~t>0$ is polynomial of degree atmost 2 such that\\\\
$ a= \frac{-\sum_{i=1}^{\eta_1} \mu(x^1_{i,1})+5\sum_{i=1}^{\eta_2} \mu(x^2_{i,1})-3\sum_{i=1}^{\eta_3} \mu(x^3_{i,1})-\sum_{i=1}^{\eta_4} \mu(x^4_{i,1})+[2c^{(1)}-2d^{(1)}
	-3c^{(2)}+d^{(2)}+d^{(3)}+c^{(4)}]}{2\kappa}.$ \\\\
$ b= \frac{2\sum_{i=1}^{\eta_1} \mu(x^1_{i,1})-12\sum_{i=1}^{\eta_2} \mu(x^2_{i,1})+6\sum_{i=1}^{\eta_3} \mu(x^3_{i,1})+4\sum_{i=1}^{\eta_4} \mu(x^4_{i,1})+[-5c^{(1)}+5d^{(1)}
	+7c^{(2)}-2d^{(2)}+c^{(3)}-3d^{(3)}-3c^{(4)}]}{2\kappa}.$

  \end{theorem}
  \begin{proof}
  Let $\kappa=4$. By the similar proof of Theorem \ref{Tm3}, $C_\phi \in B(L^{2}(\mu))$ is a $1$-quasi-$3$-isometry if and only if 
 \begin{align*}
 	\begin{split}
 		-\mu(x_1)+\mu(x_2)-3\mu(x_3)+3\mu(x_4)+\sum_{i=1}^{\eta_1} \mu(x^1_{i,1})-3\sum_{i=1}^{\eta_2} \mu(x^2_{i,1})+
 		3\sum_{i=1}^{\eta_3} \mu(x^3_{i,1})-\\
 		\sum_{i=1}^{\eta_4} \mu(x^4_{i,1})-
 	c^{(1)}+d^{(1)}	+2c^{(2)}-d^{(2)}-c^{(3)}=0
 	\end{split}
  \end{align*} 
\begin{align*}
	\begin{split}
	3\mu(x_1)-\mu(x_2)+\mu(x_3)-3\mu(x_4)-\sum_{i=1}^{\eta_1} \mu(x^1_{i,1})+\sum_{i=1}^{\eta_2} \mu(x^2_{i,1})-
	3\sum_{i=1}^{\eta_3} \mu(x^3_{i,1})+\\
	3\sum_{i=1}^{\eta_4} \mu(x^4_{i,1})-
	c^{(2)}+d^{(2)}	+2c^{(3)}-d^{(3)}-c^{(4)}=0
	\end{split}
\end{align*}
\begin{align*}
	\begin{split}
		-3\mu(x_1)+3\mu(x_2)-\mu(x_3)+\mu(x_4)+3\sum_{i=1}^{\eta_1} \mu(x^1_{i,1})-\sum_{i=1}^{\eta_2} \mu(x^2_{i,1})+
		\sum_{i=1}^{\eta_3} \mu(x^3_{i,1})-\\
		3\sum_{i=1}^{\eta_4} \mu(x^4_{i,1})-
	c^{(1)}-c^{(3)}	+d^{(3)}+2c^{(4)}-d^{(4)}=0
	\end{split}
\end{align*}
 	\begin{align*}
 	\begin{split}
 	\mu(x_1)-3\mu(x_2)+3\mu(x_3)-\mu(x_4)-3\sum_{i=1}^{\eta_1} \mu(x^1_{i,1})+3\sum_{i=1}^{\eta_2} \mu(x^2_{i,1})-\sum_{i=1}^{\eta_3} \mu(x^3_{i,1})+\\
 	\sum_{i=1}^{\eta_4} \mu(x^4_{i,1})+
 	2c^{(1)}-d^{(1)}-c^{(2)}-c^{(4)}+d^{(4)}=0.
 	\end{split}
 	\end{align*}
 	
  By solving the preceding system of linear equations, it follows that there exists a  polynomial $W(x)=ax^2+bx+t$ with $t>0$ of degree not exceeding two and values of $a$ and $b$ in ,  satisfies the necessary and sufficient condition for the operator $C_\phi$ to be a 1-quasi-3-isometry.
  \end{proof} 
 Next we study the completion problem for  $k$-quasi-$m$-isometric composition operators acting on directed graphs that contain a single circuit element; the problem reduces, in essence, to the  completion problem for unilateral weighted shifts in Proposition \ref{p1}.
 
 \begin{theorem}
 Let $m\in\mathbb{N}$, $k\in \mathbb{Z^+}$ such that $m\geq2$ and assume that $\{ b_n \}_{n=1}^m\subset (0, \infty).$ Then there exists a discrete measure space $(X, \mathcal{A}, \mu)$ together with a measurable transformation $\phi$:X$\rightarrow X$ satisfying condition \eqref{equ1} with $\kappa = 1$ and $\eta=1$ such that $\mu(x_{1,n})=b_n, ~ n\in J_{[1, m]}$. Under these conditions, the associated composition operator $C_\phi$ acts as a $k$-quasi-$(m+2)$-isometry. Furthermore, when 
 $m=1$, this construction yields the existence of a $k$-quasi-2-isometric composition operator defined in an analogous manner.
 \end{theorem}
\begin{proof}
Assume that $\{ b_n \}_{n=1}^m\subset (0, \infty).$ Given that $m\in\mathbb{N}$, $k\in \mathbb{Z^+}$ such that $m\geq2$. Now choose  $\{ b_n \}_{n=m+1}^{k+m}\subset (0, \infty)$ in such a way that there exists a polynomial $ w(x) \in \mathbb{R}[x]$ of degree $ m $ such that
$w(n) = b_{k+n}, \text{for}~ n \in J_{[1, m]},$
 and $w(n)>0, \text{for}~ n \in J_{[m+1, \infty]}$. Assume that $\kappa = 1$ and $\eta=1$,  $X=\{x_1\}\cup \{x^1_{1,j}: j\in\mathbb{N}\}$, $\mathcal{A}=P(X)$, and $\mu(x^1_{1,k+j})=w(j), j\in\mathbb{N}$,  $\mu(x^1_{1,j})=b_j, j \in J_{[1, k]}$, and $\phi(x)= Par(x), x\in X$, then $\{\mu(x^1_{1,j+k+1})\}_{j=0}^\infty$ is polynomial in $j$ of degree $m$. Then by \cite[Corollary 2.7]{DSP}, $C_\phi$ is $k$- quasi -$(m+2)$-isometric composition operator and it is the $k$- quasi -$(m+2)$-isometric  completion of the initial sequuence $\{ b_n \}_{n=1}^m\subset (0, \infty).$
 Next assume that $m=1$. Then we can choose arbitrary constant sequence $\{ b_n \}_{n=2}^\infty\subset (0, \infty)$ and a constant real polynomial $w(x)$ such that $w(j)=b_{k+j}, j\in \mathbb{N}$. As in the previous case, there exist a positive measure $\mu$ satisfies that $\mu(x^1_{1,k+j})= w(j), j\in \mathbb{N}$. So,  $C_\phi$ is a $k$- quasi -$2$-isometric  completion of the initial sequuence $\{ b_1 \}\subset (0, \infty).$
\end{proof}

\section{$k$-quasi-$m$-isometry completion problem for Weighted composition operators.}
Let $\pi\in L^{\infty}(\mu)$ and let $\phi$ be a nonsingular measurable map as
described in (\ref{equ1}).  
For each $p\in\mathbb{N}$, define
$$
\pi^{2}_{p}(x)
= \pi^{2}(x)\, (\pi\circ\phi)^{2}(x)\,(\pi\circ\phi^{2})^{2}(x)\cdots
(\pi\circ\phi^{p-1})^{2}(x).
$$

When $p\ge \kappa$, the conditional expectation of $|\pi_{p}|^{2}$ 
 is determined by the atoms of the $\sigma$-subalgebra $\phi^{-p}(\mathcal{F})$ of $\mathcal{F}$.
Thus,
$$
E_{p}(|\pi_{p}|^{2})(x)=
\begin{cases}
	K^{r}_{i,j+p}, & \text{if } x=x^{r}_{i,j+p}, r\in  J_{[1, \kappa]}, i\in  J_{[1, \eta_{r}]},
\ j\in\mathbb{N}, \\[10pt]
K^{r}_{p}, & \text{if } x=x^{s}_{i,j}\ \text{for some } s\in J_{[1, \kappa]},
\ \Phi_{2}(p+r)=\Phi_{2}(s+j),\ j\in J_{[1, p]},\\[2pt]
& \ i\in J_{[1,\eta_{s}]}, \quad \text{or if } x=x_{\Phi_{2}(p+r)} .
\end{cases}
$$
Here,
$$
K^{r}_{i,j+p}=|\pi_{p}|^{2}(x^{r}_{i,j+p}),
$$
and
$$
K^{r}_{p}=
\frac{
	|\pi_{p}|^{2}(x_{\Phi_{2}(p+r)})\,\mu(x_{\Phi_{2}(p+r)})
	+ \displaystyle\sum_{j=1}^{p}
	\sum_{\substack{s=1\\\Phi_{2}(p+r)=\Phi_{2}(s+j)}}^{\kappa}
	\sum_{i=1}^{\eta_{s}}
	|\pi_{p}|^{2}(x^{s}_{i,j})\,\mu(x^{s}_{i,j})}
{\mu(x_{\Phi_{2}(p+r)})
	+ \displaystyle\sum_{j=1}^{p}
	\sum_{\substack{s=1\\\Phi_{2}(p+r)=\Phi_{2}(s+j)}}^{\kappa}
	\sum_{i=1}^{\eta_{s}}
	\mu(x^{s}_{i,j})}.
$$
Since $E_{p}(|\pi_{p}|^{2})$ is $\phi^{-p}(\mathcal{F})$-measurable, there exists an
$\mathcal{F}$-measurable function $F_{p}$ on $X$ such that
$$
E_{p}(|\pi_{p}|^{2}) = F_{p}\circ\phi^{p},
$$
where
$$
F_{p}(x)=
\begin{cases}
	K^{r}_{i,j+p}, & \text{if } x=x^{r}_{i,j},r\in  J_{[1, \kappa]}, i\in  J_{[1, \eta_{r}]},\ j\in\mathbb{N}, \\[10pt]
	K^{r}_{p}, & \text{if } x=x_{r},\ r\in J_{[1,\kappa]}.
\end{cases}
$$
Then a weighted composition operator $W$ acting on $L^{2}(\mu)$ which is induced by a
nonsingular measurable map $\phi$ together with an essentially bounded weight
function $\pi$ is said to be a $k$-quasi-$m$-isometry \cite{DSP} if it satisfies the equation
$$
\sum_{p=0}^{m} (-1)^{p}
\binom{m}{p}\, W^{*(p+k)} W^{\,p+k}
= 
\sum_{p=0}^{m} (-1)^{p}
\binom{m}{p}\, h_{p+k} F_{p+k}
= 0.
$$
\begin{lemma}\label{l2} 
  \cite[Theorem 2.13]{DSP} Let $m\geq 2$, $k\in\mathbb{Z_+}$, $\pi$ be a positive esssentially bounded measurable function on $X$ and \eqref{equ1} hold. Then the weighted composition operator  $ W \in \mathcal{B}(L^2(\mu))$ induced by  $\phi$ and $\pi$ is $k$-quasi-$m$-isometry 	 if and only if $  \{\pi_k^2(x^{r}_{i,k+j+1})\mu(x^{r}_{i,k+j+1})\}_{j=0}^\infty $ is a polynomial in $j$ of degree at most $m-1$ for every $r\in  J_{[1, \kappa]}, i\in  J_{[1, \eta_{r}]}$ and  
	$	\Sigma_{p=0}^{m} (-1)^p \begin{pmatrix} m \\
		p \end{pmatrix} h_{p+k}F_{p+k}(x_r)= 0 ~~ \textrm{for all} ~~r \in J_{[1,\kappa]} .$
\end{lemma}

\begin{theorem}
	Let $m,\kappa,k \in \mathbb{N}$ with $\kappa>m\ge 2$, and let $\pi> 0$ is an essentially bounded measurable function on $X$. Suppose that $\eta_r$, $X$, $\mathcal{F}$, and $\phi$ are as described in \eqref{equ1}. Assume that $M\in (0,\infty)$, and for each $r\in J_\kappa$  a family of polynomials 
	$\{w_i^{\,r}\}_{i=1}^{\eta_r}$ of degree at most $m-1$, together with a sequence 
	$\{a^{\,r}_{i,j}\}_{j=1}^{\infty}\subset (0,\infty)$ such that  
	$$
	w_i^{\,r}(j)=\pi_k^2(x^{r}_{i,k+j})a^{\,r}_{i,k+j},~~ 
	\textrm{for all},~~ r\in  J_{[1, \kappa]}, i\in  J_{[1, \eta_{r}]},\; j\in\mathbb{N}.
	$$
	Moreover, suppose that
	\begin{equation}\label{eqn10}
	\max_{r\in J_\kappa}\Bigg\{	\sum_{i=1}^{\eta_r} \pi^{2}\!\left(x_{r}\right)a^{\,r}_{i,1},\;	\sup_{i\in J_{\eta_r},\, j\in\mathbb{N}} \frac{\pi^{2}\!\left(x^{r}_{i,j}\right)a^{\,r}_{i,j+1}}
	{a^{\,r}_{i,j}}\Bigg\}\leq M .
	\end{equation}

	Then there exists a measure $\mu$ on $\mathcal{F}$ satisfying:
	\begin{enumerate}
		\item $\mu\!\left(x^{r}_{i,j}\right)=a^{\,r}_{i,j}$ for every $ r\in  J_{[1, \kappa]}, i\in  J_{[1, \eta_{r}]}$, and $j\in\mathbb{N}$;
		\item the weighted composition operator $W\in B(L^{2}(\mu))$ generated by  $\phi$ and $\pi$ is a $k$-quasi-$m$-isometry.
	\end{enumerate}
	In addition, there exists $c\in\mathbb{R}$ such that for each $t>c$ there is an  another measure $\mu_t$ having the above properties and it is of the form 
	$$
	\mu_t(x)=
	\begin{cases}
		\mu(x)+t, & \text{if } x=x_r \text{ for some } r\in J_{[1, \kappa]},\\[4pt]
		\mu(x),   & \text{otherwise}.
	\end{cases}
	$$
\end{theorem}

\begin{proof}
	Let $\pi\in L^{\infty}(\mu)$ and let $\phi$ be a nonsingular measurable
	function satisfies (\ref{equ1}).  
	For  $p\in\mathbb{N}$, recall the definition
	$$
	\pi_p^{2}(x)=\pi^{2}(x)(\pi\circ\phi)^{2}(x)\cdots (\pi\circ\phi^{p-1})^{2}(x).
	$$
	If $p\ge \kappa$, the conditional expectation $E_p(|\pi_p|^{2})$ is defined as follows
	$$
	E_p(|\pi_p|^{2}) = F_p\circ\phi^{p},
	$$
	where $F_p$ is a $ \mathcal{F}$- measurable function defined by
	$$
	F_p(x)=
	\begin{cases}
		K^{r}_{i,j+p}, & x=x^{r}_{i,j},\ r\in  J_{[1, \kappa]}, i\in  J_{[1, \eta_{r}]},\
		j\in\mathbb{N},\\[4pt]
		K^{r}_{p},     & x=x_r,\ r\in J_{[1, \kappa]}.
	\end{cases}
	$$

	Define the measure on $\mathcal{F}$ by
	$$
	\mu(x^{r}_{i,j})=a^{r}_{i,j},\qquad 
	r\in  J_{[1, \kappa]}, i\in  J_{[1, \eta_{r}]},\ j\in\mathbb{N}.
	$$
	
	Observe that (\ref{eqn10}) ensures the boundedness of $W$. Moreover, the sequence\\
	  $\{ \pi_{k}^{2}(x^{r}_{i,k+j+1})\mu(x^{r}_{i,k+j+1})\}_{j=0}^\infty$ is a polynomial in $j$ of degree atmost $m-1$. Consequently, the first condition in Lemma~\ref{l2} is satisfied.
	
		Next we check the  boundary condition at the atoms  $x_{r}, r\in J_{[1, \kappa]}$ as follows:
	
	Fix  $r \in J_{[1, \kappa]}$ and define
	$$
	y_{j}:=h_{j+k}F_{j+k}(x_{r}), \qquad j\ge 0.
	$$
	Then the condition of Lemma \ref{l2} becomes
	\begin{equation}\label{eqn11}
		\sum_{p=0}^{m}(-1)^{p}\binom{m}{p}\, y_{p+k}=0	.
	\end{equation}
	
	\noindent
	Consider the block
	$$
	\mathbf{Y}^{(k)}=(y_{k},y_{k+1},\dots,y_{k+m})^{T}.
	$$
	Then \eqref{eqn11} for \(k,k+1,\dots,k+m\) gives the homogeneous system
	$$
	A_{m}\,\mathbf{Y}^{(k)}=\mathbf{0},
	$$
	where $A_{m}$ is the Toeplitz matrix containing the binomial coefficients of 
	\eqref{eqn11}.
	
	\medskip
	\noindent
Then the nullspace of $A_{m}$ consists precisely of vectors obtained from the values of polynomials of degree at most $m-1$.  
	Consequently,$\{y_j\}_{j=0}^\infty$  must be a polynomial sequence of degree $\le m-1$, which is exactly equivalent to the boundary condition
	$$
	\sum_{p=0}^{m}(-1)^{p}\binom{m}{p}\, h_{p+k}F_{p+k}(x_{r})=0.
	$$
	Since both conditions in Lemma \ref{l2} are satisfied, 
	implies that the weighted composition operator $W$ is a $k$-quasi-$m$-isometry
	on $L^{2}(\mu)$.
	
Next consider the perturbation of the measure as follows; \\
\noindent
	Let $c\in\mathbb{R}$ and $t>c$ such that $\mu(x)+t> 0$ .
	Define
	$$
	\mu_t(x)=
	\begin{cases}
		\mu(x)+t, & x=x_r,\ r\in J_{[1, \kappa]},\\[3pt]
		\mu(x), & \text{otherwise}.
	\end{cases}
	$$
	Since the translation of the measure at the root atoms doesnot alter the sequence condition for 
	$\{\pi_k^{2}(x^{r}_{i,k+j+1}) \mu(x^{r}_{i,k+j+1})\}_{j\ge0}$ and preserves the
	finite-difference condition at each $x_r$,
	Lemma \ref{l2} remain valid for $\mu_t$.  
	Hence $W$ is also a $k$-quasi-$m$-isometry on $L^{2}(\mu_t).$
	
\end{proof}

\medskip
\begin{theorem}
	Let $m\in\mathbb N$, $k\in\mathbb Z^+$ with $m\ge 2$, and let $\{b_n\}_{n=1}^m\subset(0,\infty)$.  
	Let $\pi:X\to\mathbb C$ be a positive bounded measurable function. Then there exists a discrete measure space $(X,\mathcal A,\mu)$ and a measurable map $\phi:X\to X$ satisfying condition \eqref{equ1} with $\kappa=1$ and $\eta=1$, such that
	$
	\mu(x^1_{1,n})=b_n,\quad n\in J_{[1, m]},
	$
	and the associated weighted composition operator $W$ induced by $\phi$ and $\pi$
	is a $k$-quasi-$(m+1)$-isometry.  
\end{theorem}

\begin{proof}

	Let $\kappa=1$, $\eta=1$, and set
	$$
	X=\{x_1\}\cup\{x^1_{1,j}:j\in\mathbb N\},\qquad \mathcal A=\mathcal P(X).
	$$
	Define $\phi:X\to X$ by  $\phi(x^1_{1,j})=x^1_{1,j-1}$ for $j\ge2$, $\phi(x^1_{1,1})=x_1$ and $\phi(x_1)=x_1$; these choices satisfy \eqref{eqn1} with $\kappa=1,\eta=1$.
	
	Extend the given finite sequence $\{b_n\}_{n=1}^m$ by choosing positive numbers $\{b_n\}_{n=m+1}^{\infty}$ so that there exists a real polynomial $w\in\mathbb R[x]$ of degree $m$ satisfying
	$$
	w(n)=\pi_k^{2}(x^{1}_{1,k+n})b_{k+n}\quad\text{for }n\in J_{[1,\infty)}.
	$$
    Since $\pi$ is positive and bounded, it is nowhere negative, and in particular it is nonzero on every point where we will define $\mu$. Define the atomic measure $\mu$ by
$$
\mu(x^1_{1,n})= b_n, \quad \text{for }
n\in \mathbb{N}$$
	and assign an arbitrary positive mass to the remaining finitely many atoms (e.g.\ to $x_1$) so that $\mu$ is a positive measure on $X$.

Then by Lemma \ref{l2}, $W$ is a $k$-quasi-$(m+1)$-isometry completion of the initial weight $\{b_n\}_{n=1}^m\subset(0,\infty)$.
	
	\end{proof}
\begin{remark}
	
	Moreover, when $m=1$ the same construction yields a $k$-quasi-$2$-isometric weighted composition operator. In this case choose $w$ to be a degree-$1$ polynomial (or a positive constant). So, by Lemma~\ref{l2} the operator $W$ is a $k$-quasi-$2$-isometry.
	
\end{remark}


\end{document}